\documentclass[11pt,a4paper]{amsart}
\usepackage{amssymb,amsmath,amsthm}

\newtheorem{theorem}{Theorem}[]

\newtheorem{lemma}{Lemma}

\renewcommand{\mod}{\, \textrm{mod }}

\newcommand{\etc}{,\dots,}

\newcommand{\R}{\mathbb{R}}
\newcommand{\Z}{\mathbb{Z}}
\newcommand{\eps}{\varepsilon}
\newcommand{\al}{\alpha}
\newcommand{\be}{\beta}
\newcommand{\ga}{\gamma}

\newcommand{\conv}{\mathrm{conv} \,}

\begin{document}

\title{Small subset sums} 

\author{Gergely Ambrus}
\author{Imre B\'ar\'any}
\author{Victor Grinberg}

\begin{abstract}
Let $\|.\|$ be a norm in $\R^d$ whose unit ball is $B$. Assume that $V\subset B$ is a finite set of cardinality $n$, with $\sum_{v \in V} v=0$. We show that for every integer $k$ with $0 \le k \le n$, there exists a subset $U$ of $V$ consisting of $k$ elements such that $\|\sum_{v \in U}v\| \le\lceil d/2 \rceil$. We also prove that this bound is sharp in general. We improve the estimate to $O(\sqrt d)$ for the  Euclidean and the max norms. An application on vector sums in the plane is also given.
\end{abstract}

\subjclass[2010]{52A40, 05B20}
\keywords{Vector sums, Steinitz theorem, normed spaces.}

\maketitle

\section{Definitions, notation, results}\label{introd}
We consider the real $d$-dimensional vector space $\R^d$ with a norm $\|.\|$ whose unit ball is $B$. For a finite set $U\subset \R^d$, $|U|$ stands for the cardinality of $U$, and $s(U)$ for the sum of the elements of $U$, so $s(U)=\sum_{u \in U}u$, and $s(\emptyset)=0$ of course.

\medskip
In 1914 Steinitz \cite{St_20} proved that, in the case of the Euclidean norm, for every finite set $V \subset B$ with $|V|=n$ and $s(V)=0$, there exists an ordering $v_1, \ldots, v_n$ of the vectors in $V$ such that {\em all partial sums} have norm at most $2d$, that is
\[
\max_{k=1\etc n}\Big\| \sum_{1}^k v_i \Big\| \le 2d.
\]
It is important here that the bound $2d$ does not depend on $n$, the size of~$V$. Steinitz's result implies that for every norm and every finite $V \subset B$ with $s(V)=0$ there is an ordering along which all partial sums are bounded by a constant that depends only on $B$. Let $S(B)$ denote the smallest such constant for a given norm with unit ball $B$, and set $S(d)= \sup S(B)$ where the supremum is taken over all norms in $\R^d$. The best known bounds on $S(d)$ are: $S(d)\le d$, proved by Sevastyanov~\cite{Sev}, and by Grinberg and Sevastyanov \cite{Gr_Se}, and $S(d)\ge \frac {d+1} 2 $, which is shown by an example coming from the $\ell_1$ norm~\cite{Gr_Se}. For specific norms, stronger results may hold. In particular, for $\ell_2$ and $\ell_{\infty}$, it is conjectured that the right order of magnitude of $S(B)$ is $\sqrt{d}$ -- although not even $o(d)$ is known.

\medskip
Steinitz's result immediately implies that for every finite set $V \subset B$ with $s(V)=0$ and every integer $k$, $0\le k \le |V|$, there is a subset $U \subset V$ such that $|U|=k$ and $\|s(U)\|$ is not greater than a constant depending only on $d,B,k$, for instance $S(B)$ is such a constant. Let $T(B,k)$ be the smallest constant with this property, set $T(B)=\sup_k T(B,k)$, and $T(d)=\sup T(B)$ where the supremum is taken over all norms in $\R^d$. It is evident that $T(B,k)\le k$.

\medskip
In this paper we investigate $T(B,k), T(B)$ and $T(d)$. Here come our main results. First, the estimate for general norms.

\begin{theorem} \label{thm1}
Let $B$ be the unit ball of an arbitrary norm on $\R^d$. For any finite set $V \subset B$ with $s(V)=0$, and for any $k\le |V|$, there exists a subset $U \subset V$ with $k$ elements, so that
\[
\| s(U) \| \le \left \lceil \frac d 2 \right \rceil.
\]
In other words, $T(d) \le \left \lceil \frac d 2 \right \rceil$.
\end{theorem}

\begin{theorem} \label{thm2}
For every $d \ge 1$, there exists a norm in $\R^d$ with unit ball $B$, so that $T(B,k)= \left \lceil \frac d 2 \right \rceil$ for infinitely many values of $k$. Also, $ T(B,k) =k$ for all $k \le \left\lfloor \frac d 2 \right\rfloor$.
\end{theorem}

Theorems~\ref{thm1} and \ref{thm2} imply that $T(d)=\left \lceil \frac d 2 \right \rceil$ for all integers $d \ge 1$.

\medskip
One expects that for specific norms better estimates are valid. We have proved this in some cases. The unit ball of the norm $\ell_p^d$ will be denoted by~$B_p^d$. We have the following results in the cases $p=1,2,\infty$.

\begin{theorem}\label{ell1}
$\frac d2 \le T(B_1^d)\le \left\lceil \frac d2 \right\rceil$.
\end{theorem}

\begin{theorem}\label{ell2}
$\frac 12 \sqrt {d+2} \le T(B_2^d) \le \frac {1 + \sqrt 5} 2\sqrt d$
\end{theorem}

\begin{theorem}\label{ellinf}
$\frac 13 {\sqrt d} \le T(B_{\infty}^d) \le O(\sqrt d)$
\end{theorem}

We mention that in Theorems~\ref{ell2} and \ref{ellinf} the order of magnitude is the same as the conjectured value of the Steinitz constant.

\medskip\noindent
{\bf Remark 1.} Note that there is a "complementary" symmetry here. Namely, for every $U \subset V$, $s(U)=-s(V\setminus U)$, hence $\|s(U)\|=\|s(V\setminus U)\|$, and the cases $k$ and $n-k$ are symmetric.  Hence, we may assume $k \le n/2$.

\medskip
When establishing Helly-type theorems for sums of vectors in a normed plane, B\'ar\'any and Jer\'onimo-Castro proved the following result \cite[Lemma~5]{Ba_Jer}, which matches our scheme: {\em Given $6$ vectors in the unit ball of a normed plane whose sum is 0, there always exist $3$ among them, whose sum has norm at most 1.} In fact, this statement served as the starting point for our current research. An application of  Theorem~\ref{thm1} implies an extension of one of the Helly-type results \cite[Theorem 3]{Ba_Jer}, which we formulate slightly differently and prove in the last section.

\begin{theorem}\label{Helly}
Let $k \ge 2$ be a positive integer, and  $n = m (k-1) +1$ for some $m \ge 1$. Assume $B$ is the unit ball of a norm in $\R^2$, $V \subset B$ is of size $n$ and $\|s(V)\| \le 1$. Then $V$ contains a subset $W$ of size $k$ such that $\|s(W)\|\le 1$.
\end{theorem}

\medskip

\section{Proof of Theorem 1}

We are to consider linear combinations $\sum_{v \in V}\al(v)v$ of the vectors in $V$. The coefficients $\al(v)$ form a vector  $\al \in \R^V$, the $|V|$--dimensional real vector space whose coordinates are indexed by the elements of $V$. Define the convex polytope
\[
P(V,k)=\Big\{\al \in \R^V: \sum_{v \in V}\al(v)v=0,\;\sum_{v \in V}\al(v)=k,\; 0\le \al(v)\le 1\;(\forall v\in V)\Big\}.
\]
$P(V,k)$ is non-empty as $\al(v) \equiv k/n$ lies in it (here $n=|V|$). From now on let $\al$ denote a fixed vertex of $P(V,k)$. The basic idea is to choose $U$ to be the set of vectors from $V$ that have the $k$ largest  coefficients $\al(v)$.  This works directly when $d$ is odd, and some extra care is needed for even $d$.

\medskip
We note first that $P(V,k)$ is determined by $d+1$ linear equations and $2n$ inequalities for the coefficients $\al(v)$, so at a vertex at most $d+1$ coefficients are strictly between $0$ and $1$. Define $U_1=\{v \in V: \al(v)=1\}$ and $Q=\{v \in V: 0<\al(v)<1\}$. Set $q=\sum_{v \in Q} \al(v)$, $q$ is an integer since $q+|U_1|=k$. Split now $Q$ into two parts, $E$ and $F$, so that $|E|=q$ and $E$ contains the vectors with the $q$ largest coefficients in $Q$, and $F$ the rest (ties broken arbitrarily). Then $U=U_1 \cup E$ has exactly $k$ elements and
\begin{eqnarray*}
s(U)&=&\sum_{v \in U_1}v+ \sum_{v \in E}v\\
    &=&\sum_{v \in V}\al(v)v+\sum_{v \in E}(1-\al(v))v -\sum_{v \in F}\al(v)v.
\end{eqnarray*}
Here $\sum_{v \in V}\al(v)v=0$, so by the triangle inequality
\[
||s(U)||\le \sum_{v \in E}(1-\al(v)) +\sum_{v \in F}\al(v).
\]

\medskip
The average of the coefficients in $Q$ is $a:=q/|Q|$. Thus, the average of the coefficients is at least $a$ in $E$, and it is at most $a$ in $F$. Consequently, the last sum is maximal when $\al(v)=a$ for all $v \in Q$:
\begin{equation*}\label{thm1_2}
||s(U)||\le q(1-a) +(|Q|-q)a=\frac 2{|Q|}\, q \,(|Q|-q) \le \frac {|Q|}2\,.
\end{equation*}
This finishes the proof when $d$ is odd as $|Q| \le d+1$, and also when $d$ is even and $|Q| \le d$.

\medskip
We are left with the case when $d$ is even and $|Q|=d+1$. The vectors in $Q$ are linearly dependent, so there is a non-zero $\be \in \R^V$ with $\be(v)=0$ when $v \notin Q$ such that $\sum_{v \in Q}\be(v)v=0$. We can assume that $\sum_{v \in Q}\be(v)\le 0$. Then $\sum_{v \in V}(\al(v)+t\be(v))v=0$ for every $t \in \R$. Choose $t>0$ maximal so that $0\le \ga(v)=\al(v)+t\be(v))\le 1$ for every $v \in V$. This means that, for some $v^* \in Q$, $\ga(v^*)=0$ or $1$.

\smallskip
Assume for the time being that $q \le (d+1)/2$.

\smallskip
Suppose first that $\ga(v^*)=0$. This time we split $Q^*: =Q \setminus v^*$ again into $E$ and $F$ so that $|E|=q$ and $E$ contains the vectors from $Q^*$ with the $q$ largest coefficients. Note that $\sum_{v \in Q^*}\ga(v) \le \sum_{v \in Q} \al(v)=q$ and that $|Q^*| =d$, so the average $a^*$ of $\ga(v)$ over $Q^*$ is at most $q/d$. We use again $U=U_1 \cup E$ and we have, the same way as before,
\[
||s(U)||\le \sum_{v \in E}(1-\ga(v)) +\sum_{v \in F}\ga(v).
\]
The right hand side is maximal again if every $\ga(v)$ equals their average $a^*$, hence
\[
||s(U)|| \le q(1-a^*)+(d-q)a^*=q+(d-2q)a^* \le q+(d-2q)\frac qd \le \frac d 2,
\]
because $d$ is even so $q \le (d+1)/2$ implies $2q \le d$. Thus, $||s(U)|| \le d/2$.

\medskip
The case when $\ga(v^*)=1$ is similar: this time $v^*$ is added to $U_1$, $Q^*=Q \setminus v^*$ is split into $E$ and $F$ with $|E|=q-1$  so that $E$ contains the vectors with the largest $q-1$ coefficients. Now $\sum_{v \in Q^*}\ga(v) \le \sum_{v \in Q} \al(v)-1=q-1$, and thus the average $a^*$ of $\ga(v)$ over $Q^*$ is at most $(q-1)/d$. As above, we are led to  the  inequality
\[
||s(U)|| \le (q-1)(1-a^*)+(d-(q-1))a^*=(q-1)+(d-2(q-1))a^*.
\]
Using that $d-2(q-1)\ge 0$ and $a^* \le (q-1)/d$,  we conclude that  $||s(U)||\le d/2 - 2/d < d/2$.

Finally we consider the case $q > (d+1)/2$. Since $s(U)=-s(V\setminus U)$, we may consider the complementary problem of finding $\widetilde{U} \subset V$ with $n-k$ elements so that $||s(\widetilde{U})||\le \lceil d/2 \rceil$. Clearly, $(1-\al(v))_{v \in V}$ is a vertex of $P(V,n-k)$, for which $\sum_{v \in Q}(1-\al(v))<(d+1)/2$. Thus, the previous argument may be applied to construct $\widetilde{U}$.
\qed

\bigskip
The same proof yields a stronger statement.

\begin{theorem}
Let $W \subset B$ finite. Suppose that for some $k \le |W|$ and $w_0 \in \mathrm{conv} \, W$, the vector $w_0$ may be expressed as $w_0 = \sum_{w \in W} \alpha(w) w$ with $\sum_{w \in W} \alpha(w) =1$ and $\alpha(w) \in [0, \frac 1 k]$ for every $w$.\footnote{In the original, printed version of the article, this condition on the coefficients is erroneously missing. We are indebted to Tung Nguyen for pointing out the mistake.} Then there is a subset $U \subset W$ of cardinality $k$, so that
\[
\| s(U)  - k w_0 \| \le \left \lceil \frac d 2 \right \rceil.
\]
\end{theorem}

The proof is the same as above, except that instead of the convex polytope $P(V,k)$, we consider the coefficient vectors $\alpha: W \rightarrow [0,1]$ satisfying
\[
\sum_{w \in W} \alpha(w) w = kw_0 \mbox{ and }\sum_{w \in W} \alpha(w) =k.
\]
The condition of the theorem ensures that this set is a non-empty convex polytope. The rest of the argument is unchanged.

\medskip\noindent
{\bf Remark 2.}  For later reference we record the fact that the linear dependence $\al$ defines the sets $U_1$ and $Q$, and if $|Q|=d+1$, then the new linear dependence $\ga$ defines $v^* \in Q$ and $Q^*$. Note that this works for even and odd $d$, we only need $|Q|=d+1$. For later use we define
\begin{equation}\label{AC}
A=\{v \in V: \ga(v)=1 \} \mbox{ and }C=\{v \in V: 0<\ga(v)<1\}.
\end{equation}

\medskip

\section{Proof of Theorem~\ref{thm2}}

We are going to use the following fact.  If the unit ball of a norm $\|.\|$ is the convex hull of the vectors $v_1\etc v_m, -v_1\etc -v_m \in \R^d$, then for every vector $x\in \R^d$,
\[
\|x\|=\min \Big\{\sum_1^m |a_i|\;:\sum_1^m a_iv_i=x\Big\}.
\]

Let $e_1,\ldots,e_d$ be the standard basis vectors of $\R^d$, and set $e_0=-\sum_1^d e_i$. We define $V$ to be $s$ copies of $\{e_0,e_1,\ldots,e_d\}$, where $s\ge 1$ is an integer. The unit ball is set to be $B=\conv \{V,-V\}$. Let $k<n=s(d+1)$ be a positive integer congruent to $\left \lceil \frac d2\right \rceil \mod (d+1)$. We claim that for every $k$-element subset $U$ of $V$, $\|s(U)\| \ge \left \lceil \frac d 2 \right\rceil$.

\medskip
Assume that $U$ contains $b_i$ copies of $e_i$ for every $i$, so $k=\sum_0^d b_i$. We have to estimate the norm of the vector $v = \sum_0^d b_i e_i$. Assume that
\[
v = \sum_0^d a_i e_i
\]
for some $a_i \in \R$. Then $\sum _0^d (b_i - a_i) e_i=0$. Since the only linear dependence of the vectors $e_0\etc e_d$ is $x \sum_0^d e_i =0$ for some constant $x \in \R$, we obtain that  $a_i = b_i -x$ for every $i$. Set
\[
f(x) := \sum_0^d |b_i-x|,
\]
Then $\|v\|= \min f(x)$ by the fact from the beginning of this section. We are going to estimate $f(x)$. Since $b_i \in \Z$ for every $i$, the function $f(x)$ is piecewise linear on $\R$ (it is affine on all intervals $(q,q+1)$ for $q \in \Z$). Therefore, there exists $c \in \Z$ so that the minimum of $f(x)$ is attained at $c$.

\medskip
The facts $k=\sum_0^d b_i \equiv \lceil d/2 \rceil \mod (d+1)$ and $c \in \Z$ imply that $\sum_0^d (b_i-c) \equiv \lceil d/2 \rceil \mod (d+1)$. Thus,
\[
\left\lceil \frac d 2 \right\rceil \le \Big|\sum_0^d (b_i-c)\Big| \le \sum_0^d |b_i-c|,
\]
hence, $\| v \| \ge \lceil d/2 \rceil$.

\medskip
We show next that $T(B,k)=k$ when $1\le k < \lceil d/2 \rceil$. The unit ball $B$ is the same as above and $V=\{e_0\etc e_d\}$. Assume $U \subset V$ with $|U|=k$ and $\|s(U)\| <k$. Add $\lceil d/2 \rceil-k$ vectors from $V \setminus U$ to $U$ to obtain a subset $W$ of $\lceil d/2 \rceil$ elements. Every addition increases the norm of the sum by at most one (because of the triangle inequality), so we get $\|s(W)\| \le  \|s(U)\|+\lceil d/2 \rceil-k < \lceil d/2 \rceil$, contrary to what was established above. Thus $T(B,k) \ge k$, while $T(B,k)\le k$ follows from the triangle inequality. \qed

\medskip
Further examples showing $T(B,k)= \lceil d/2 \rceil$ will be given in the next section.

\bigskip\noindent
{\bf Remark 3.} We mention that for large enough $n$, there is no vector set that works simultaneously for all $k$ with $d/2 \le k \le n-d/2$. This follows from Steinitz's theorem: let $v_1\etc v_n$ be the ordering where all partial sums lie in $dB$. Then necessarily two partial sums, with at least $d/2$ summands whose cardinalities differ by at least $d/2$, are close to each other: a standard volume estimate shows that their distance is bounded above by $4dn^{-1/d}$. Then their difference, which is a $k$-sum with some  $d/2 \le k \le n-d/2$, must be small.


\section{The $\ell_1$ norm, proof of Theorem \ref{ell1}}

The upper bound follows from Theorem \ref{thm1}. For the lower bound let $V$ consist of $e_1\etc e_d$ and $d$ copies of $\frac 1d e_0$ (with the same notation as in the previous section). Assume $U \subset V$ has exactly $d$ elements. If $U$ contains $p$ vectors out of $e_1\etc e_d$, then $s(U)$ has $p$ coordinates equal to $\frac pd$ and $d-p$ coordinates equal to $\frac pd-1$. Thus $\|s(U)\|_1=\frac 1d(p^2+(d-p)^2)$. The last expression is minimal when $p=\lfloor \frac d2\rfloor$. The minimum equals $\frac d2$ when $d$ is even and $\frac d2 + \frac 1{2d}$ when $d$ is odd. This is slightly better (for $d$ odd) than the stated lower bound. \qed

\bigskip
This example shows that $T(B_1^d)=T(B_1^d,d) = d/2$ for even $d$. A small modification gives further examples implying $T(B_1^d, k) = d/2$ for even $d$ and for all $k\ge d$. Namely, given $d\ge 1$ and $k\ge d$, let $V$ consist of the vectors $e_1, \ldots, e_d$, and $2k-d$ copies of $\frac 1{2k-d}e_0$.  Then $V\subset B_1^d$ and  $s(V) =0$. It is not hard to check that this shows $T(B_1^d, k) = d/2$ for every $k \ge d$ ($d$ is even).


\section{The $\ell_2$ norm, proof of Theorem \ref{ell2}}

In this section, $\|.\|$ stands for the Euclidean  norm. For the upper bound we will need two lemmas. The first is Lemma 2.2 in Beck's paper~\cite{beck}. A similar result is given in \cite[Theorem 4.1]{Ba_08}. The second is a Steinitz type statement.

\begin{lemma}\label{parall} Let $Q \subset B_2^d$ be finite, and $\al:Q \to[0,1]$.
Then there exists $\eps: Q \to \{0,1\}$ such that
$\|\sum_{v\in Q} (\eps(v)-\al(v))v\|\le \sqrt d/2$.
\end{lemma}

\begin{lemma}\label{order} Assume that $V \subset B_2^d$ is a finite set and $\|s(V)\|=\sigma$. Then there exists an ordering $v_1,\ldots,v_n$ of the elements of $V$, such that, for all $h\le n$,
\[
\Big\|\sum_1^h v_i \Big\| \le \sqrt{\sigma^2+h}.
\]
\end{lemma}

\begin{proof} Choose $v_1 \in V$ arbitrarily. For $h \ge 2$, we select $v_h$ inductively. We set $S_h= \sum_1^h v_i$. Assume that $\|S_{h-1}\| \le \sqrt{\sigma^2 + h-1}$, and set $W = V \setminus \{ v_1, \ldots, v_{h-1} \}$. We consider three cases.

\smallskip
{\bf Case 1.} If $\|S_{h-1}\| \le \sigma-1$, then choose $v_h \in W$ arbitrary: $\|S_h\|\le \sigma$ holds  by the triangle inequality.

\smallskip
{\bf Case 2.} If $\|S_{h-1}\|\ge \sigma$, then by the assumption $\|S\| = \sigma$, there exists a vector $v_h \in W$, for which $\langle S_{h-1}, v_h \rangle \le 0$. Therefore,
\[
\| S_h \|^2 = \|S_{h-1} + v_h\|^2 \le \|S_{h-1}\|^2 + \|v_h\|^2 \le  (\sigma^2 + h-1 )+ 1 = \sigma^2 + h.
\]

\smallskip
{\bf Case 3.} If $\sigma -1 < \|S_{h-1}\| < \sigma$, define $\varepsilon= \sigma-\|S_{h-1}\|$, so $0< \varepsilon < 1$ and $\eps \le \sigma$. Then
\[
\sum_{v \in W} \langle v, S_{h-1} \rangle = \langle S_h - S_{h-1}, S_{h-1} \rangle  \le \sigma(\sigma - \varepsilon) - (\sigma - \varepsilon)^2 = \varepsilon (\sigma - \varepsilon).
\]
Thus, there exists $v_h \in W$, for which $ \langle v_h, S_{h-1} \rangle \le  \varepsilon (\sigma - \varepsilon)$. Then
\begin{align*}
\| S_h \|^2 &= \|S_{h-1} + v_h\|^2 \le (\sigma - \varepsilon)^2 + 2\, \varepsilon (\sigma - \varepsilon) + 1\\ &= \sigma^2 + 1 - \varepsilon^2<\sigma^2 +h.
\qedhere
\end{align*}
\end{proof}

\begin{proof}[Proof of Theorem~\ref{ell2}] For the lower bound let $V$ be the set of vertices of a regular simplex inscribed in $B_2^d$. Then $s(V)=0$. Let $U\subset V$ have $\left \lceil \frac d2 \right \rceil$ elements. A routine computation shows that $\|s(U)\|$ equals $\frac {\sqrt{d+2}}{2}$ when $d$ is even and $\frac {d+1}{2\sqrt d} > \frac {\sqrt{d+2}}{2}$ when $d$ is odd. This implies the lower bound $T(B_2^d) \ge \frac {\sqrt {d+2}}2$.

\smallskip
For the upper bound we have to prove the existence of $U \subset V$ with $|U|=k$ and $\|s(U)\| \le \frac{1+\sqrt 5}2 \sqrt d$. From the proof of Theorem 1 recall the definition of $P(V,k)$ and its vertex $\al \in \R^V$ and $U_1=\{v\in V:\al(v)=1\}$ and $Q=\{v \in V: 0<\al(v)<1\}$. Here $|Q|\le d+1$.

\smallskip
If $|Q|=0$, then $|U_1|=k$ and $s(U_1)=0$, so we can set $U=U_1$. The case $|Q|=1$ is impossible because the sum of all $\alpha(v)$ is an integer. From now on we assume that $2 \le |Q|$ implying
$|U_1|+1\le k \le |U_1|+|Q|-1$. Using Lemma 1 for $\alpha$ restricted to $Q$ we
find $\eps:Q \to \{0,1\}$ such that $\|\sum_{v\in Q} (\eps(v)-\alpha(v))v||\le \sqrt d/2$.

\smallskip
Define $W=U_1 \cup \{v \in Q : \eps(v)=1\}$, then $W$ has the properties that
$\|s(W)\| \le \sqrt d /2$ and $||W|-k|\le d$. Because of the complementary symmetry,
we can assume that $k\le |W| \le k+d$. Set $h =|W|-k$.
Then Lemma 2 applies to $W$: writing $\sigma=\|s(W)\|$  we have $\sigma \le \sqrt d/2$ and so the
elements of $W$ can be ordered as $w_1, w_2,\ldots$ so that $\|\sum_1^m w_i\|\le \sqrt{\sigma^2+m}$
for every $m$. In particular, with $m = h \le d$, $\|\sum_1^h w_i\|\le \sqrt{\sigma^2+h}\le \sqrt{d/4+d}$.
Then for $U= W\setminus \{w_1,...w_h\}$, we have $|U|=k$ and $\|s(U)\|\le \frac {1+\sqrt 5}2 \sqrt d$.
\end{proof}


\section{The $\ell_{\infty}$ norm, proof of Theorem \ref{ellinf}}

Here, $\|.\|$ denotes the maximum norm. We need two lemmas again, the first is similar to Lemma~\ref{parall}.

\begin{lemma}\label{parallinf} If $C \subset B_{\infty}^d$ consists of $d$ linearly independent vectors, then for every point $z$ of the parallelotope $P=\sum_{v \in C} [0,v]$, there is a vertex $u$ of $P$ with $\|z-u\|_\infty=O(\sqrt d)$.
\end{lemma}

This is a result of Spencer \cite[Corollary 8]{Spe_six}, and also of Gluskin \cite{Gl} whose work relies on that of Kashin \cite{Ka}. Spencer's proof gives the estimate $\|z-u\|\le 6\sqrt d$. The linear independence condition is only needed to ensure that $P$ is a parallelotope, and so its vertices are of the form $s(D)=\sum_{v \in D}v$ for some subset $D \subset C$.

The next statement is the (weaker) analogue of Lemma~\ref{order} for the $l_\infty$ norm. Note that we require the set $W$ to contain only a few vectors. The proof is longer and it uses Chobanyan's transference theorem (for the $\ell_{\infty}$ norm) so we postpone it to Section~\ref{stein}.

\begin{lemma}\label{lstein} Assume $W \subset B_{\infty}^d$, $|W|=m \le 5d$, and $\|s(W)\|_\infty=O(\sqrt d)$. Then there is an ordering $w_1\etc w_m$ of the vectors in $W$ such that
\[
\max_{h=1\etc m} \big\|\sum_1^h w_i \big\|_\infty=O(\sqrt d).
\]
\end{lemma}

\begin{proof}
[Proof of Theorem~\ref{ellinf}] The lower bound uses Hadamard matrices and is given in \cite{Ba_08}.

For the upper bound we assume, rather for convenience than necessity, that the set $V\subset \R^d$ is in general position, for instance, no $d$ vectors from $V$ are linearly dependent. The general case follows from this by a limit argument. We assume further that $|V|=n>5d$ since for $n\le 5d$ the result is a consequence of  Lemma~\ref{lstein}. Set $m=\lfloor n/(2d)\rfloor$.

We are going to define linear dependencies $\ga_i$, for $i=1,2\etc m-1$ so that the sets
\[
A_i=\{v \in V: \ga_i(v)=1\},\; C_i=\{v \in V: 0<\ga_i(v)<1\}
\]
satisfy the conditions
\[
A_i\subset A_{i+1},\; (2i-1)d\le |A_i| < h_i:=\sum_{v \in V}\ga_i(v)\le 2di,\; |C_i|=d.
\]

The construction is recursive and is similar to how $\al$ and $\ga \in \R^V$ were constructed. For $i=1$ we take an arbitrary vertex $\al$ of the convex polytope $P(V,2d)$, then $|Q|=d+1$ (because of the general position assumption) and $d \le |U_1| <2d$ follows. We construct $\ga$ as specified in Remark 2 and~\eqref{AC}. Then define $\ga_1=\ga$, set $A_1=\{v \in V: \ga_1(v)=1\},\; C_1=\{v \in V: 0<\ga_1(v)<1\}$. General position implies that $|C_1|=d$ and then $d\le |A_1|< h_1=\sum_{v\in V}\ga_1(v) \le 2d$.

\smallskip
Assume next that $\ga_1\etc \ga_i$ have been constructed ($1<i<m-1$), and the sets $A_j,C_j$ for $j\le i$ satisfy the required conditions. Define the convex polytope
\[
P_{i+1}= \left\{\al \in P(V,2d(i+1)): \al(v)=1\; (\forall v\in A_i)\right\}
\]
We check that $P_{i+1}$ is non-empty. As $|A_i|<h_i\le 2di$, the linear dependence $\al =\ga_i+t(\mathbf{1}-\ga_i)$ lies in $P_{i+1}$ for a suitable $t$, we only have to check that $0< t < 1$ as this implies $0\le \al(v)=\ga_i(v)+t(1-\ga_i(v))\le 1$. To fulfill the condition $\sum_{v \in V}\al(v)=2d(i+1)$, we must set
\[
t=\frac {2d(i+1)-h_i}{n-h_i}=1-\frac {n-2d(i+1)}{n-h_i}.
\]
Thus $0< t <1$ indeed as $h_i \le 2di$.

\smallskip
Next, let $\al_{i+1}$ be a fixed vertex of $P_{i+1}$. The method recorded in Remark 2 gives another linear dependence $\ga_{i+1}$ with
$|C_{i+1}|=d$. $A_i \subset A_{i+1}$ by the construction. All $v \in V$ with $\al_{i+1}(v)=1$ are in $A_{i+1}$, and there are at least $2d(i+1)-d$ of them. Thus $(2i+1)d \le |A_{i+1}|$. Further $|A_{i+1}|<h_{i+1}$ follows since $\ga_{i+1}(v)=1$ for every $v \in A_{i+1}$ and $h_{i+1}\le 2d(i+1)$ because $h_{i+1}=\sum_{v \in V}\ga_{i+1}(v)\le \sum_{v \in V}\al_{i+1}(v)= 2d(i+1)$.

\smallskip
The construction is almost finished, as a last step we define $A_0=C_0=\emptyset$.

\smallskip
We use Lemma~\ref{parallinf} next. The parallelotope $P:=\sum_{v \in C_i}[0,v]$ contains the point $-s(A_i)$, since $0=s(A_i)+\sum_{v \in C_i}\ga_i(v)v$. A vertex of $P$ is of the form $s(D)=\sum_{v \in D}v$, where $D$ is a subset of $C_i$. By Lemma~\ref{parallinf}, there is a $D_i\subset C_i$ such that the vertex $s(D_i)$ is at distance $O(\sqrt d)$ from $-s(A_i)$. Thus the vector $z_i=s(A_i\cup D_i)$ is short, namely, $\|z_i\|=O(\sqrt d)$. Note that by setting $D_0=\emptyset$, we have $z_0=0$ which is again of norm $O(\sqrt d)$.

For the next step of the proof we first check that the size of the symmetric difference $(A_{i+1}\cup D_{i+1})\triangle (A_i\cup D_i)$ is at most $5d$. This holds for $i=0$. For larger $i$, $D_{i+1}$ and $A_{i+1}$ are disjoint, and $A_{i+1}$ contains $A_i$, so the symmetric difference is the same a
$X \triangle D_i$, where $X=(A_{i+1}\setminus A_i)\cup D_{i+1}$. Here $|A_{i+1} \setminus A_i|<3d$, and both $D_i$ and $D_{i+1}$ have at most $d$ elements, which gives the upper bound $5d$.

Now $z_i-s(D_i)+s(X)=z_{i+1}$. Thus, adding at most $5d$ vectors from $B_{\infty}^d$ to $z_i$ one arrives at $z_{i+1}$, and both $z_i,z_{i+1}$ are short. Define
\[
W=\{-u: u \in D_i \setminus X\}\bigcup (X \setminus D_i).
\]
Then $W$ is a subset of $B_{\infty}^d$, of at most $5d$ elements, such that $s(W)=\sum_{w \in W}w=z_{i+1}-z_i$. Thus $\|s(W)\|=O(\sqrt d)$. By applying Lemma~\ref{lstein} to $W$ we get an ordering $w_1\etc w_m$ such that every partial sum along this ordering is $O(\sqrt d)$. Then for every $h=1\etc m$.
\[
\|z_i+\sum_1^h w_j \|\le \|z_i \| +\|\sum_1^h w_j \| =O(\sqrt d).
\]

In the original problem we have to show that for every $k \le n$ there is a set $U \subset V$ of size $k$ with $\|s(U)\|=O(\sqrt d)$. This is clear when $k$ equals the size of some $A_i\cup D_i$, but what is to be done for the other $k$? Well, such a $k$ lies between $|A_i\cup D_i|$ and $|A_{i+1}\cup D_{i+1}|$ for some $i$. Note that $z_i=s(A_i\cup D_i)$. Moreover, each sum $z_i+w_1+\ldots +w_h$ is the sum of vectors in a subset of $V$. This can be seen by induction on $h$. The case $h=0$ is clear. The induction step $h-1 \to h$ is clear again when $w_h$ does not come from $D_i$, simply one more term appears in the sum. If however $w_h$ comes from $D_i$, then it cancels the previous $-w_h$ that is a unique term in $s(A_i\cup D_i)$. So each partial sum is a subset-sum. The number of elements in these subsets increases or decreases by one when the next $w_h$ is added. Then for every $k$ between
$|A_i\cup D_i|$ and $|A_{i+1}\cup D_{i+1}|$ there is a partial sum containing exactly $k$ terms. \end{proof}

\medskip

\noindent
{\bf Remark 4.}
The above proof yields a slightly stronger statement: we construct a chain of subsets of $V$, each with sum of order of magnitude $O(\sqrt{d})$, so that the cardinality of two consecutive subsets differ by one, and the chain traverses from the empty set to $V$. We have hoped to give a better value for the Steinitz constant $S(B_2^d)$ or $S(B_{\infty}^d)$ by a suitable modification of the argument (we would need an {\em increasing} chain of subsets with the previous properties), but  our efforts have failed so far.
\medskip

\noindent
{\bf Remark 5.}
A simpler proof may be given if one only aims for the existence a $k$-element subset with small sum. We may assume that $k \le n-d$. Starting from a vertex of $P(v, k-d)$ and using Lemma~\ref{parallinf}, similarly to the proof of Theorem~\ref{ell2}, we can construct a set $W$ so that $\|s(W)\| \le 6\sqrt{d}$, and $k - 2d \le |W| \le k$. Let $\alpha$ be the characteristic function of $W$, i. e. $\alpha(v) = 1$ if $v \in W$, and 0 otherwise. Let $l = |W|$, and set $t$ so that $l + t(n - l) = k+d$. Then $t \le 1$.

Next, consider the set $P$ of the linear dependencies $\beta: V \rightarrow  [0,1]$ with
\[\sum_{v \in V} \beta(v) v = (1 - t) s(W),\  \sum_{v \in V} \beta(v) = k + d,\ \beta(v)=1 (\forall v \in W).
\]
Then $P$ is a non-empty convex polytope, since $\alpha + t (\mathbf{1} - \alpha)$ satisfies all the above conditions. Take an arbitrary a vertex of $P$. As before, invoking Lemma~\ref{parallinf}, we find a set $Y$ so that $\|s(Y) - (1 - t) s(W)\|_\infty = O(\sqrt{d})$, and $||Y|- (k+d)| \le d$. Furthermore, the construction implies that $W \subset Y$. Hence,
\[k - 2d \le |W| \le k \le |Y| \le k+ 2d,\]
and  $\|s(W)\| = O(\sqrt{d})$ as well as $\|s(Y)\| = O(\sqrt{d})$. We finish the proof by applying Lemma~\ref{lstein} to the set $Y \setminus W$.

\medskip
{\bf Remark 6.} The above proofs translate for arbitrary norms as long as the analogues of Lemmas~\ref{parall} and ~\ref{order} (or Lemmas~\ref{parallinf} and ~\ref{lstein})  may be established.

\section{Proof of Lemma~\ref{lstein}}\label{stein}

For this lemma it is natural to use Chobanyan's transference theorem~\cite{Cho} (see also \cite{Ba_08}), which connects Steinitz's theorem with sign assignments to vectors in a sequence.

Assume $v_1\etc v_n$ is a sequence of vectors from the unit ball $B$ of an arbitrary norm on $\R^d$. It is proved in \cite{BG} that there are signs $\eps_1\etc \eps_n =\pm 1$ such that
\begin{equation}\label{signs}
\max_{k=1\etc n} \Big\|\sum_{i=1}^k\eps_iv_i\Big\| \le 2d-1.
\end{equation}
This is a general bound that does not depend on $n$ and the norm. But better estimates are valid for specific norms and some (small) values of $n$. For fixed $B$ and $n$ let $F(B,n)$, the {\sl sign sequence constant of $B$}, be defined as the smallest number that one can write on the right hand side of (\ref{signs}), and let $F(B)=\sup_n F(B,n)$. It is quite easy to see for instance that $F(B_2^d,n) \le \sqrt n$ for all $n$ (but we don't need this). What we need is a result of Spencer \cite[Theorem 1.4]{Spe_max}:

\medskip\noindent
{\bf Fact 1.} $F(B_{\infty}^d,d) \le K \sqrt d$ where $K$ is a universal constant.
\medskip

Chobanyan's transference theorem \cite{Cho} says that, for every norm with unit ball $B$, $S(B)\le F(B)$, that is, the Steinitz constant is at most as large as the sign sequence constant. We need a slightly stronger variant, so we define $S(B,n)$ as the smallest number $R$ such that for every set $V \subset B$ with $s(V)=0$ and $|V|=n$ there is an ordering $v_1\etc v_n$ of the elements in $V$ such that
\[
\max_{k=1\etc n} \Big\|\sum_{i=1}^k v_i\Big\| \le R.
\]
Of course, $S(B)=\sup_n S(B,n)$.
Here comes the stronger version of Choba{\-}nyan's theorem, and comes without proof as the proof is identical with the original one.

\begin{theorem}\label{chob} For every norm in $\R^d$ with unit ball $B$, $S(B,n) \le F(B,n)$.
\end{theorem}

Theorem~\ref{chob} and Fact 1 imply the following.

\medskip\noindent
{\bf Fact 2.} Given $V \subset B_{\infty}^d$ with $|V|=m$ where $m \le 5d$ and $s(V)=0$, there is an ordering $v_1\etc v_m$ of $V$ such that $\max_{h=1\etc m}\|\sum_1^h v_i\|_\infty \le K_1\sqrt d$, where $K_1$ is a universal constant.

\begin{proof}
We note first that for $m\le d$ this follows directly from Fact 1 and Theorem~\ref{chob} with $K_1=K$. For $m\ge d$, take the natural embedding of $\R^d$ into $\R^m$, the set $V$ lies in the $\ell_{\infty}$ unit ball of $\R^m$. Apply Fact 1 and Theorem~\ref{chob}
there, and you get an ordering of $V$ in $\R^d$ along which all partial sums have norm at most $K\sqrt{m}\le K\sqrt{5d}$. Thus Fact 2 holds with $K_1=\sqrt 5 K$.
\end{proof}

\begin{proof}[Proof of Lemma~\ref{lstein}] We need a concrete bound on $\|s(W)\|_\infty$, so suppose that $\|s(W)\|_\infty \le K_2\sqrt d$. For $w \in W$ define $w^*=w-\frac 1m s(W)$. Then $\|w^*\|_\infty \le \|w\|_\infty+\frac 1m \| s(W)\|_\infty \le 2$ as $s(W)$,
being the sum of $m$ vectors from $B_{\infty}^d$, has norm at most $m$. Further, $\sum_{w \in W} w^*=0$ and $W \subset 2B_{\infty}^d$. By Fact 2
there is an ordering $w_1\etc w_m$ of the vectors in $W$ such that for every $h$
\[
\Big\|\sum_1^h w_i^*\Big\|_\infty \le 2K_1 \sqrt d.
\]
We check that $\sum_1^h w_i = \sum_1^h w^*_i + \frac hm s(W)$ and so for every $h$
\[
\Big\|\sum_1^h w_i \Big\|_\infty \le \Big\|\sum_1^h w^*_i \Big\|_\infty + \Big\|s(W) \Big\|_\infty  \le 2K_1\sqrt d+ K_2\sqrt d=O(\sqrt d). \qedhere
\]
\end{proof}

\medskip

\section{An application: proof of Theorem~\ref{Helly}}

We proceed by induction on $m$. For $m=1$, the assertion is clearly true. For the induction step $(m-1) \to m$ let $V \subset B$ with $|V|=(k-1)m+1$ and $\|s(V)\|\le 1$. Set $v_0 = -s(V)$ so $\|v_0 \| \le 1.$ Define $V_0= V \cup \{ v_0\}$. Then $V_0 \subset B$ and $s(V_0)=0$. So by Theorem~\ref{thm1}, there exists a subset $U$ of $V_0$ of size $k$, with $\| s(U) \| \le 1$. We are done if $v_0 \notin U$. So suppose that $v_0 \in U$. Then $v_0\notin W : =V \setminus U$, and $\|s(W)\|\le 1$ because
\[
s(U) = -s(W).
\]
Here $W$ is of size $(m-1)(k-1) +1$, so the induction hypothesis implies that $W$ contains a subset $U$ of size $k$ with $\| s(U) \| \le 1$.
\qed

\medskip

We mention finally that Theorem~\ref{Helly} is equivalent to the following Helly type statement. If $V \subset B$ and $|V|=(k-1)m+1$, and $\|s(U)\| >1$ for every set $U \subset V$ of size $k$, then $\|s(V)\|>1$.

\bigskip
{\bf Acknowledgments.} Research of the first author was supported by SCIEX grant 12052. Research of the second author was partially supported by ERC Advanced Research Grant 267165 (DISCONV), and by Hungarian National Research Grant K 83767.

\vspace{0.4 cm}

 {\sc Gergely Ambrus}\\[1mm]
  {\footnotesize R\'enyi Institute of Mathematics}\\[-1mm]
  {\footnotesize Hungarian Academy of Sciences}\\[-1mm]
  {\footnotesize PO Box 127, 1364 Budapest, Hungary, and }\\[1mm]
  {\footnotesize \'Ecole Polytechnique F\'ed\'erale de Lausanne}\\[-1mm]
  {\footnotesize EPFL SB  MATHGEOM DCG Station 8}\\[-1mm]
  {\footnotesize CH-1015 Lausanne, Switzerland}\\[1mm]
  {\footnotesize e-mail: {\tt ambrus@renyi.hu}}\\

 {\sc Imre B\'ar\'any}\\[1mm]
  {\footnotesize R\'enyi Institute of Mathematics}\\[-1mm]
  {\footnotesize Hungarian Academy of Sciences}\\[-1mm]
  {\footnotesize PO Box 127, 1364 Budapest, Hungary, and }\\[1mm]
  {\footnotesize Department of Mathematics}\\[-1mm]
  {\footnotesize University College London}\\[-1mm]
  {\footnotesize Gower Street, London WC1E 6BT, England}\\[1mm]
  {\footnotesize e-mail: {\tt barany@renyi.hu}}\\

{\sc Victor S. Grinberg}\\[1mm]
{\footnotesize 5628 Hempstead Rd, Apt 102, Pittsburgh}\\[-1mm]
  {\footnotesize PA 15217, USA}\\[1mm]
  {\footnotesize e-mail: {\tt victor\_\,grinberg@yahoo.com}\\

\end{document}